\documentclass[11pt]{article}
\usepackage[a4paper]{geometry}
\usepackage{amsfonts, amsmath, amssymb, amsthm, graphicx, caption, authblk, url, multirow, makecell, framed, float, xcolor, enumitem, tikz, hyperref}
\usepackage{algorithm,algorithmicx,algpseudocode,mathrsfs}
\usetikzlibrary{fit,shapes.arrows}
\theoremstyle{definition}
\newtheorem{theorem}{Theorem}
\newtheorem{definition}{Definition}
\newtheorem{example}{Example}

\newtheorem{lemma}{Lemma}
\newtheorem{corollary}{Corollary}
\theoremstyle{remark}

\newcommand{\h}[1]{\underline{\mathbf{#1}}}

\newcommand*{\mybox}[1]{%
  \framebox{\raisebox{0cm}[0.5\baselineskip][0.05\baselineskip]{%
    \hbox to 0.10cm {\hss#1\hss}}}\hspace{0.05cm}}

\begin{document}
\title{Cyclic Equalizability Characterized by Parikh Vectors}
\author[1]{Sarunyu Thongjarast\thanks{\texttt{thong125@mit.edu}}}
\author[2]{Sarit Pasiphol\thanks{\texttt{6632234621@student.chula.ac.th}}}
\author[2]{Suthee Ruangwises\thanks{\texttt{suthee@cp.eng.chula.ac.th}}}
\affil[1]{Massachusetts Institute of Technology, Cambridge, MA, USA}
\affil[2]{Department of Computer Engineering, Faculty of Engineering, Chulalongkorn University, Bangkok, Thailand}
\date{}
\maketitle

\begin{abstract}
Cyclic equalizability is a notion introduced by Shinagawa and Nuida in 2025, in the study of card-based cryptography. Informally, a collection of words is cyclically equalizable if, by inserting the same letters at the same positions in all words, they can be transformed into words that are cyclic shifts of one another. Shinagawa and Nuida showed that two binary words of equal length are cyclically equalizable if and only if they have the same Hamming weight. They also posed the problem of characterizing cyclic equalizability over larger alphabets. In this paper, we completely characterize cyclic equalizability for two words over an arbitrary finite alphabet by proving that two words are cyclically equalizable if and only if they have the same Parikh vector.

\end{abstract}

\section{Introduction}
\emph{Card-based cryptography} is a research area that studies how to perform secure computation using a deck of physical cards. Since the seminal \emph{five-card trick} of den Boer~\cite{5card} in 1990, this field has developed a rich theory connecting combinatorics, algebra, and computation~\cite{chosen,formal}. In recent years, new connections have emerged between card-based protocols and combinatorics on words, suggesting that structural properties of words can play a fundamental role in understanding such protocols.

A central operation in card-based cryptography is the \emph{random cut}, which cyclically shifts a sequence of face-down cards by an unknown offset. As a result, sequences that differ only by a cyclic shift become indistinguishable. This naturally leads to the notion of cyclic equivalence (also known as conjugacy) of words, which has been extensively studied in combinatorics on words~\cite{combi3,combi2,combi}.

Motivated by this observation, Shinagawa and Nuida~\cite{cyclic} introduced the notion of cyclic equalizability in 2025. Informally, a collection of words is \emph{cyclically equalizable} if, by inserting the same letters at the same positions in all words, they can be transformed into words that are cyclic shifts of one another. This notion arises naturally in the analysis of card-based protocols, where inserting letters corresponds to inserting additional cards to the sequence without revealing the existing cards.

For binary words of equal length, Shinagawa and Nuida showed that cyclic equalizability admits a simple characterization: two words are cyclically equalizable if and only if they have the same Hamming weight. They also posed several open problems, including characterizing cyclic equalizability for words over larger alphabets such as $\{0,1,2\}$ (and more generally, over arbitrary finite alphabets).

In combinatorics on words, the Parikh vector of a word records the number of occurrences of each letter. Two words are said to be Abelian equivalent if they have the same Parikh vector, i.e. if they are permutations of each other. Related notions have been studied in combinatorics on words, including conjugacy and other structural properties of words under insertion operations~\cite{insertion}, and subword counting in circular words (conjugacy classes)~\cite{subword}.

\subsection{Our Contribution}
In this paper, we completely characterize cyclic equalizability for two words over an arbitrary finite alphabet. In particular, we prove that two words over a finite alphabet are cyclically equalizable if and only if they have the same Parikh vector.

As an immediate consequence, we obtain a new interpretation of cyclic equalizability in terms of Abelian equivalence. Our result strictly generalizes the binary case of Shinagawa and Nuida \cite{cyclic}, and resolves their question for the ternary alphabet as a special case. Also, the result has implications for card-based cryptography: it characterizes when two encodings can be made indistinguishable under random cuts by inserting additional cards.

\section{Preliminaries}
Let $\Sigma$ be a finite alphabet. A \emph{word} over $\Sigma$ is a finite sequence of letters from $\Sigma$. The set of all words over $\Sigma$ is denoted by $\Sigma^*$, and the empty word is denoted by $\varepsilon$. For a word $w \in \Sigma^*$, its length is denoted by $|w|$. If $w = a_0 a_1 \cdots a_{n-1}$, where $a_i \in \Sigma$, then $|w| = n$.

\subsection{Parikh Vectors and Abelian Equivalence}
For a letter $a \in \Sigma$, let $|w|_a$ denote the number of occurrences of $a$ in a word $w$. The \emph{Parikh vector} of a word $w \in \Sigma^*$ is the vector
\[
\Psi(w) = (|w|_a)_{a \in \Sigma}.
\]

Two words $u, v \in \Sigma^*$ are said to be \emph{Abelian equivalent} if $\Psi(u) = \Psi(v)$, i.e. if they contain the same number of occurrences of each letter. Equivalently, two words are Abelian equivalent if one is a permutation of the other.

\subsection{Cyclic Equivalence}
Let $w = a_0 a_1 \cdots a_{n-1} \in \Sigma^n$ with $n \ge 1$. For an integer $r \in \mathbb{Z}_n$, the \emph{cyclic shift} of $w$ by $r$ positions is the word
\[
w^{(r)} = a_r a_{r+1} \cdots a_{n-1} a_0 a_1 \cdots a_{r-1},
\]
where the indices are taken modulo $n$.

Two words $u, v \in \Sigma^*$ of equal length are said to be \emph{cyclically equivalent} (or \emph{conjugate}), denoted by $u \sim v$, if $v = u^{(r)}$ for some integer $r$. We also call such $r$ an \emph{offset} of $u$ and $v$.

\subsection{Simultaneous Insertion}
Let $k \ge 2$ and let $w_1, \dots, w_k \in \Sigma^n$ be words of equal length $n$. A \emph{simultaneous insertion} is an operation that transforms each word $w_i = a_{i,0} a_{i,1} \cdots a_{i,n-1}$ into the word
\[
u_{0} a_{i,0} u_{1} a_{i,1} \cdots u_{n-1} a_{i,n-1} u_{n},
\]
where each $u_{j}$ is a word in $\Sigma^*$. In other words, we insert the same strings at the same positions in all words.

For example, if $k=3$, $w_1 = 121$, $w_2 = 334$, $w_3 = 135$, and we insert $u_0 = 14$, $u_1 = 5$, $u_2 = \varepsilon$ (an empty string), $u_3 = 56$, then $w_1$, $w_2$, and $w_3$ become
\[
\mathbf{\underline{14}}1\mathbf{\underline{5}}21\mathbf{\underline{56}}, \quad \mathbf{\underline{14}}3\mathbf{\underline{5}}34\mathbf{\underline{56}}, \quad \text{and} \quad
\mathbf{\underline{14}}1\mathbf{\underline{5}}35\mathbf{\underline{56}},
\]
respectively, where the bold underlined letters indicate the inserted letters.

\subsection{Cyclic Equalizability}
We now define the central notion studied in this paper.

Let $k \ge 2$ and let $w_1, \dots, w_k \in \Sigma^*$ be words of equal length. We say these words are \emph{cyclically equalizable} if there exist words $w'_1, \dots, w'_k \in \Sigma^*$ such that
\begin{enumerate}
    \item Each $w'_i$ is obtained from $w_i$ by a simultaneous insertion, and
    \item The words $w'_1, \dots, w'_k$ are cyclically equivalent.
\end{enumerate}

For example, the words $123$ and $132$ are cyclically equalizable, since they can be transformed into
\[
12\mathbf{\underline{1}}3 \quad \text{and} \quad 13\mathbf{\underline{1}}2,
\]
where the bold underlined letters indicate the inserted letters. The transformed words are cyclically equivalent.

Similarly, the words $12344$ and $42431$ are cyclically equalizable, since they can be transformed into cyclically equivalent words
\[
123\mathbf{\underline{12}}44\mathbf{\underline{24}} \quad \text{and} \quad 424\mathbf{\underline{12}}31\mathbf{\underline{24}},
\]
where the inserted letters are again shown in bold and underlined.

\section{Known Properties of Cyclic Equalizability}
In this section, we review basic properties of cyclic equalizability established by Shinagawa and Nuida~\cite{cyclic}.

The following theorem shows that cyclic equalizability is invariant under simultaneous insertions.

\begin{theorem}[{\cite[Theorem~1]{cyclic}}] \label{shinagawa1}
Let $w_1, \dots, w_k \in \Sigma^*$ be words of equal length, and let $w'_1, \dots, w'_k$ be the words obtained from them by a simultaneous insertion. Then, $w_1, \dots, w_k$ are cyclically equalizable if and only if $w'_1, \dots, w'_k$ are cyclically equalizable.
\end{theorem}

The next theorem gives a complete characterization of cyclic equalizability for two binary words.

\begin{theorem}[{\cite[Theorem~2]{cyclic}}] \label{shinagawa2}
Let $u, v \in \{0,1\}^*$ be binary words of equal length. Then, $u$ and $v$ are cyclically equalizable if and only if $|u|_1 = |v|_1$.
\end{theorem}

\section{Cyclic Equalizability of Two Words}
In this section, we prove a complete characterization of cyclic equalizability for two words over an arbitrary finite alphabet.

\begin{theorem} \label{mainthm}
Let $\Sigma$ be a finite alphabet and let $u,v \in \Sigma^*$ be words of equal length. Then, $u$ and $v$ are cyclically equalizable if and only if they have the same Parikh vector.
\end{theorem}

We first prove the necessity. Suppose $u$ and $v$ are cyclically equalizable. Simultaneous insertions add the same letters to both words at each step, and hence preserve Abelian equivalence. Since cyclically equivalent words are Abelian equivalent, it follows that the original words $u$ and $v$ are also Abelian equivalent.

It remains to prove the sufficiency, namely that if two words $u,v \in \Sigma^*$ have the same Parikh vector, then they are cyclically equalizable.

\subsection{Reduction to Distinct Letters and Normalization}
We begin with a sequence of reductions that simplify the problem so that it suffices to prove when $u = 0\,1\,\cdots\,(n-1)$ and $v = \pi(0)\,\pi(1)\,\cdots\,\pi(n-1)$ for an arbitrary permutation $\pi \in S_n$.

\begin{lemma} \label{reduction}
It suffices to prove the statement for words in which all letters are pairwise distinct.
\end{lemma}

\begin{proof}
Suppose $u$ and $v$ are words of size $n$ with the same Parikh vector. For each letter $x \in \Sigma$ appearing $k \geq 2$ times in $u$ (and hence also in $v$), replace its occurrences by $k$ distinct letters $x_1, \dots, x_k$ arbitrarily. Perform this replacement in both $u$ and $v$ for every letter appearing multiple times. For instance, $u = acbcac$ and $v = cbaacc$ become
\[
    \tilde{u} = a_1 c_1 b c_2 a_2 c_3 \quad \text{and} \quad \tilde{v} = c_1 b a_1 a_2 c_2 c_3.
\]

This yields two new words $\tilde{u}, \tilde{v}$ of length $n$ over a larger alphabet, in which all letters are pairwise distinct. By construction, $\tilde{u}$ and $\tilde{v}$ are permutations of each other.

If $\tilde{u}$ and $\tilde{v}$ are cyclically equalizable, then identifying all letters $a_1, \dots, a_k$ back to $a$ shows that $u$ and $v$ are cyclically equalizable. Therefore, it suffices to consider only the case where all letters are distinct.
\end{proof}

By Lemma~\ref{reduction}, we may assume that $u$ and $v$ consist of $n$ distinct letters. Relabel the alphabet so that the letters of $u$ are $0,1,\dots,n-1$ in this order. Since $v$ is a permutation of the same letters, it can be written as $\pi(0)\pi(1)\cdots\pi(n-1)$ for some permutation $\pi \in S_n$.

Since cyclic equalizability is invariant under relabeling of the alphabet, we obtain the following normalization.

\begin{corollary} \label{normalize}
It suffices to consider the case where
\[
u = 0\,1\,\cdots\,(n-1) \quad \text{and} \quad
v = \pi(0)\,\pi(1)\,\cdots\,\pi(n-1)
\]
for an arbitrary permutation $\pi \in S_n$.
\end{corollary}

\subsection{Useful Lemmas}
We present two lemmas that will be used in our construction in the next two subsections.

The first lemma shows that cyclic equivalence is preserved under reading a word with a step size that is coprime with the length.

\begin{definition} \label{stepdef}
Let $w = a_0a_1\cdots a_{n-1} \in \Sigma^n$ and let $p$ be an integer such that $\gcd(p,n)=1$, define the \emph{reading with step size $p$} of $w$ to be the word
\[
R_p(w)=a_0\,a_p\,a_{2p}\,\cdots\,a_{(n-1)p},
\]
where the indices are taken modulo $n$.
\end{definition}

\begin{lemma} \label{step}
Let $u,v \in \Sigma^n$ and let $p$ be an integer such that $\gcd(p,n)=1$. If $R_p(u)$ and $R_p(v)$ are cyclically equivalent with offset $d$, then $u$ and $v$ are cyclically equivalent with offset $dp$.
\end{lemma}

\begin{proof}
Let $u=a_0a_1\cdots a_{n-1}$ and $v=b_0b_1\cdots b_{n-1}$. Define a map $\varphi:\mathbb{Z}_n \to \mathbb{Z}_n$ by $\varphi(i)=ip \pmod n$ for every $i \in \mathbb{Z}_n$. Since $\gcd(n,p)=1$, the map $\varphi$ is a permutation of $\mathbb{Z}_n$.

We have
\[
R_p(u)=a_{\varphi(0)}a_{\varphi(1)}\cdots a_{\varphi(n-1)}
\quad \text{and} \quad
R_p(v)=b_{\varphi(0)}b_{\varphi(1)}\cdots b_{\varphi(n-1)}.
\]

Suppose $R_p(u)$ and $R_p(v)$ are cyclically equivalent with offset $d$. Then, $b_{\varphi(i)} = a_{\varphi(i+d)}$ for every $i \in \mathbb{Z}_n$. Since
\[
\varphi(i+d) \equiv (i+d)p \equiv ip + dp \equiv \varphi(i) + dp \pmod n,
\]
it follows that $b_{\varphi(i)} = a_{\varphi(i)+dp}$ for every $i \in \mathbb{Z}_n$, where the indices are taken modulo $n$.

Since $\varphi$ is a permutation of $\mathbb{Z}_n$, this implies $b_j = a_{j+dp}$ for every $j \in \mathbb{Z}_n$. Hence, $u$ and $v$ are cyclically equivalent with offset $dp$.
\end{proof}

The next lemma shows that cyclic equivalence is preserved under columnwise interleaving, provided that all component words have the same length and offset.

\begin{definition} \label{interleavingdef}
Let $w_1,\dots,w_m \in \Sigma^n$ be words of equal length $n$, where $w_i = w_{i,0}w_{i,1}\cdots w_{i,n-1}$ for each $i=1,2,\dots,m$. Define the \emph{columnwise interleaved word} of $w_1,\dots,w_m$ to be the word
\[
I(w_1,\cdots,w_m) = w_{1,0}w_{2,0}\cdots w_{m,0}\,
    w_{1,1}w_{2,1}\cdots w_{m,1}\,
    \cdots\,
    w_{1,n-1}w_{2,n-1}\cdots w_{m,n-1}.
\]
\end{definition}

\begin{lemma} \label{interleaving}
Let $u_1,\dots,u_m,v_1,\dots,v_m \in \Sigma^n$ be words of equal length $n$, where $u_i = u_{i,0}u_{i,1}\cdots u_{i,n-1}$ and $v_i = v_{i,0}v_{i,1}\cdots v_{i,n-1}$ for each $i=1,2,\dots,m$. Suppose for every $i \in \{1,\dots,m\}$, the word $v_i$ is a cyclic shift of $u_i$ by the same offset $\delta$, i.e. $v_{i,j} = u_{i,j+\delta}$ for all $j \in \mathbb{Z}_n$, where $j+\delta$ is taken modulo $n$. Then, the words $U = I(u_1,\cdots,u_m)$ and $V = I(v_1,\cdots,v_m)$ are cyclically equivalent.
\end{lemma}

\begin{proof}
Let $U = a_0a_1\cdots a_{mn-1}$ and $V = b_0b_1\cdots b_{mn-1}$. Under the linear indexing, the letter $u_{i,j}$ (resp. $v_{i,j}$) occurs in $U$ (resp. $V$) at position $mj + (i-1)$, i.e. $a_{mj + (i-1)} = u_{i,j}$ and $b_{mj + (i-1)} = v_{i,j}$ for every $i \in \{1,\dots,m\}$ and $j \in \mathbb{Z}_n$.

By assumption, $v_{i,j} = u_{i,j+\delta}$ for every $i,j$. Therefore, the letter at position $mj+(i-1)$ in $V$ is equal to the letter at position $m(j+\delta)+(i-1) \pmod{mn}$ in $U$. Since
\[
m(j+\delta)+(i-1) \equiv mj+(i-1)+m\delta \pmod{mn},
\]
it follows that $a_t = b_{t+m\delta}$ for every $t \in \mathbb{Z}_{mn}$, where the indices are taken modulo $mn$. Hence, $U$ and $V$ are cyclically equivalent.
\end{proof}


\subsection{Construction for Single-Cycle Case}
Recall that $u = 0\,1\,\cdots\,(n-1)$ and $v = \pi(0)\,\pi(1)\,\cdots\,\pi(n-1)$. We first consider the case where $\pi \in S_n$ consists of a single cycle. Thus, for each $t \in \mathbb{Z}_n$, there exists a unique $k \in \mathbb{Z}_n$ such that $t=\pi^k(0)$.

We will construct words $u'=u'_0u'_1\cdots u'_{n^2-1}$ and $v'=v'_0v'_1\cdots v'_{n^2-1}$ such that
\begin{enumerate}
    \item $u'$ and $v'$ can be obtained from $u$ and $v$, respectively, by a simultaneous insertion, and
    \item $u'$ and $v'$ are cyclically equivalent with offset $n+1$.
\end{enumerate}

At a high level, we partition $u'$ and $v'$ into $n$ blocks of $n$ consecutive positions. We perform insertions so that the $i$-th position of $u$ and $v$ is placed in Block $i$ of $u'$ and $v'$, for each $i \in \mathbb{Z}_n$. See Example~\ref{example1} for an illustration of this construction.

\subsubsection{Blocks and Groups}
Let $p=n+1$. Since $\gcd(p,n^2)=1$, we may consider the reading with step size $p$ on words of length $n^2$.

We divide the positions of $u'$ and $v'$ into $n$ consecutive \emph{blocks} of size $n$ according to their original order, and into $n$ consecutive \emph{groups} of size $n$ according to their reading order with step size $p$.

\begin{definition}
We partition the positions in $\mathbb{Z}_{n^2}$ into $n$ blocks and $n$ groups. For $t,g \in \mathbb{Z}_n$:
\begin{itemize}
    \item Block $t$ consists of positions
    \[
    tn,tn+1,\dots,tn+(n-1).
    \]
    \item Group $g$ consists of positions
    \[
    \varphi(gn), \varphi(gn+1), \dots, \varphi(gn+n-1),
    \]
    where $\varphi(i)=ip \pmod{n^2}$ for all $i \in \mathbb{Z}_{n^2}$ and $p = n+1$.
\end{itemize}
\end{definition}

Next, we show that each pair of a block and a group uniquely determines a position.

\begin{lemma}
For each $t,g \in \mathbb{Z}_n$, there exists a unique position in $\mathbb{Z}_{n^2}$, denoted by $f(t,g)$, that lies in Block $t$ and Group $g$. Moreover, $f$ is a bijection from $\mathbb{Z}_n^2$ to $\mathbb{Z}_{n^2}$.
\end{lemma}

\begin{proof}
For any $t,g \in \mathbb{Z}_n$, let $i \equiv t-g \pmod n$, where $i \in \mathbb{Z}_n$. We have
\[
\varphi(gn+i) \equiv (gn+i)(n+1) \equiv gn^2 + (g+i)n + i \equiv (g+i)n + i \pmod{n^2}.
\]
Therefore,
\[
\left\lfloor \frac{\varphi(gn+i)}{n} \right\rfloor \equiv g+i \equiv t \pmod n,
\]
so $\varphi(gn+i)$ lies in Block $t$. By construction, it also lies in Group $g$, so such a position exists.

For uniqueness, suppose $\varphi(gn+i')$ also lies in Block $t$, where $i' \in \mathbb{Z}_n$. Then, $g+i' \equiv t \pmod n$, which implies $i' \equiv t-g \equiv i \pmod n$ and thus $i'=i$. Therefore, the position is unique.

Hence, $f(t,g)$ is well-defined. Since blocks partition $\mathbb{Z}_{n^2}$ and groups partition $\mathbb{Z}_{n^2}$, distinct pairs $(t,g)$ correspond to distinct positions, so $f$ is injective. As $|\mathbb{Z}_n^2| = |\mathbb{Z}_{n^2}| = n^2$, it follows that $f$ is a bijection.
\end{proof}

\subsubsection{Filling Strategy}
We construct words $u'$ and $v'$ from $u$ and $v$ by a simultaneous insertion so that $u'$ and $v'$ are cyclically equivalent with offset $p$.

For each position $i \in \mathbb{Z}_{n^2}$, the letters $u'_i$ and $v'_i$ either correspond to original letters $u_j$ and $v_j$ for some $j \in \mathbb{Z}_n$, or arise from the inserted letters. We call $i$ a \emph{distinguished position} in the former case, and a \emph{non-distinguished position} in the latter case.

We first designate exactly $n$ distinguished positions. This determines the letters $u'_i$ and $v'_i$ at all distinguished positions. The remaining positions are filled so as to ensure that $u'$ and $v'$ are cyclically equivalent with offset $p$.

\subsubsection{Filling Distinguished Positions}
For each $k \in \mathbb{Z}_n$, we designate $f(\pi^k(0),\,k)$ to be a distinguished position. At this position, we set the letters
\[
u'_{f(\pi^k(0),\,k)}=\pi^k(0) \quad \text{and} \quad
v'_{f(\pi^k(0),\,k)}=\pi^{k+1}(0).
\]
Since $\pi^0(0),\pi^1(0),\dots,\pi^{n-1}(0)$ are pairwise distinct, there is exactly one distinguished position in each block (and also exactly one in each group).

\subsubsection{Filling Non-Distinguished Positions}
We now fill the remaining positions by following the reading order $\varphi(0),\varphi(1),\dots,\varphi(n^2-1)$, skipping the distinguished positions already filled.

For convenience, define $a_i := u'_{\varphi(i)}$ and $b_i := v'_{\varphi(i)}$. Note that the letters $a_0$ and $b_0$ are already filled, since $\varphi(0)=0=f(0,0)$ is a distinguished position.

For each $i=1,2,\dots,n^2-1$, if $\varphi(i)$ is a non-distinguished position, we set
\[
a_i = b_{i-1}  \quad \text{and} \quad
b_i = b_{i-1}.
\]

\begin{lemma} \label{main1}
The words $R_p(u')$ and $R_p(v')$ are cyclically equivalent with offset $1$.
\end{lemma}

\begin{proof}
Note that $R_p(u') = a_0, a_1, \cdots, a_{n^2-1}$ and $R_p(v') = b_0, b_1, \cdots, b_{n^2-1}$. We claim that $b_i = a_{i+1}$ for every $i \in \mathbb{Z}_{n^2}$, where the indices are taken modulo $n^2$.

First, if $\varphi(i+1)$ is a non-distinguished position, then by construction $a_{i+1} = b_i$, so the claim holds.

Now suppose $\varphi(i+1)$ is distinguished. Let it belong to Block $t$ and Group $g$. By construction, we have $t = \pi^g(0)$ and $\varphi(i+1) = f(t,g) = f(\pi^g(0),g)$. Also, $a_{i+1} = u'_{\varphi(i+1)} = \pi^g(0)$.

\textbf{Case 1:} $g>0$. Consider the maximum value $j$ such that $j<i+1$ and $\varphi(j)$ is distinguished. Since each group contains exactly one distinguished index, $\varphi(j)$ must be in Group $g-1$. Also, all of $\varphi(j+1), \varphi(j+2), \dots, \varphi(i)$ must be non-distinguished.

By construction, we have $\varphi(j) = f(\pi^{g-1}(0),g-1)$ and $b_j = v'_{\varphi(j)} = \pi^g(0)$. Furthermore, since $\varphi(j+1), \varphi(j+2), \dots, \varphi(i)$ are non-distinguished, we have $b_i = b_{i-1} = \cdots = b_j = \pi^g(0)$. Hence, $a_{i+1} = b_i$.

\textbf{Case 2:} $g=0$. We have $\varphi(i+1) = 0$, so $i=n^2-1$ and $a_{i+1}=0$. Consider the maximum value $j \in \mathbb{Z}_{n^2}$ such that $\varphi(j)$ is distinguished. Since each group contains exactly one distinguished index, $\varphi(j)$ must be in Group $n-1$. Also, all of $\varphi(j+1), \varphi(j+2), \dots, \varphi(n^2-1)$ must be non-distinguished.

By construction, we have $\varphi(j) = f(\pi^{n-1}(0),n-1)$ and $b_j = v'_{\varphi(j)} = \pi^n(0) = 0$. Furthermore, since $\varphi(j+1), \varphi(j+2), \dots, \varphi(n^2-1)$ are non-distinguished, we have $b_i = b_{n^2-1} = b_{n^2-2} = \cdots = b_j = 0$. Hence, $a_{i+1} = b_i$.

Therefore, $R_p(u')$ and $R_p(v')$ are cyclically equivalent with offset $1$.
\end{proof}

From Lemma~\ref{step}, we can conclude that $u'$ and $v'$ are cyclically equivalent with offset $p$.

\begin{lemma} \label{main2}
The words $u'$ and $v'$ can be obtained from $u$ and $v$, respectively, by a simultaneous insertion.
\end{lemma}

\begin{proof}
For each Block $t$, let $k$ be the unique integer in $\mathbb{Z}_{n}$ such that $t=\pi^k(0)$. Then, the distinguished position in Block $t$ is $f(t,k)$, and at that position we have
\[
u'_{f(t,k)} = t \quad \text{and} \quad
v'_{f(t,k)} = \pi(t).
\]
Also, if $i \in \mathbb{Z}_{n^2}$ is a non-distinguished position, then $u'_i = v'_i$.

Therefore, if we keep only the distinguished positions, and delete all other positions, from $u'$ we obtain
\[
0\,1\,\cdots\,(n-1)=u,
\]
while from $v'$ we obtain
\[
\pi(0)\,\pi(1)\,\cdots\,\pi(n-1)=v.
\]
Since the same letters are deleted in both words, $u'$ and $v'$ can be obtained from $u$ and $v$, respectively, by a simultaneous insertion.
\end{proof}

From Lemmas~\ref{step}, \ref{main1} and~\ref{main2}, we now obtain the desired conclusion for the single-cycle case.

\begin{theorem} \label{mainthm1}
If $\pi \in S_n$ consists of a single cycle, then the words
\[
u = 0\,1\,\cdots\,(n-1) \quad \text{and} \quad
v = \pi(0)\,\pi(1)\,\cdots\,\pi(n-1)
\]
are cyclically equalizable.
\end{theorem}

\begin{example} \label{example1}
Let $n=5$, $u=01234$, and $v=30421$. We have $\pi(0)=3$, $\pi(1)=0$, $\pi(2)=4$, $\pi(3)=2$, and $\pi(4)=1$, so $\pi=(0\,3\,2\,4\,1)$ is a single cycle. Also, we have $\varphi(i)=6i \pmod{25}$. Thus, the sequence $\varphi(0), \varphi(1), \dots, \varphi(24)$ is
\[
0,6,12,18,24,5,11,17,23,4,10,16,22,3,9,15,21,2,8,14,20,1,7,13,19,
\]
which is the order of reading with step size $p=6$. The distinguished positions are
\[
f(0,0)=0,\quad f(3,1)=17,\quad f(2,2)=10,\quad f(4,3)=21,\quad f(1,4)=7.
\]

Table~\ref{table1} shows the words $R_p(u')$ and $R_p(v')$, which are the reading with step size $p=6$ of $u'$ and $v'$, respectively. The positions are divided into five groups of size five. The distinguished positions are shown in bold underlined. One can see that $R_p(v')$ is a cyclic shift of $R_p(u')$ by one position.

Table~\ref{table2} shows the words $u'$ and $v'$ written in their original order, divided into five blocks of size five. Again, the distinguished positions are shown in bold underlined. Deleting all non-bold entries yields $u$ from $u'$ and $v$ from $v'$.

\begin{table}[H]
\centering
\caption{The reading with step size $p=6$ of $u'$ and $v'$, divided into five groups} \label{table1}
\small
\begin{tabular}{|c|ccccc|}
\hline
\multicolumn{6}{|c|}{Group 0 ($i=0,1,\dots,4$)} \\
\hline
$\varphi(i)$          & $\mathbf{\underline{0}}$ & 6 & 12 & 18 & 24 \\
$a_i=u'_{\varphi(i)}$ & $\mathbf{\underline{0}}$ & 3 & 3 & 3 & 3 \\
$b_i=v'_{\varphi(i)}$ & $\mathbf{\underline{3}}$ & 3 & 3 & 3 & 3 \\
\hline
\multicolumn{6}{|c|}{Group 1 ($i=5,6,\dots,9$)} \\
\hline
$\varphi(i)$          & 5 & 11 & $\mathbf{\underline{17}}$ & 23 & 4 \\
$a_i=u'_{\varphi(i)}$ & 3 & 3 & $\mathbf{\underline{3}}$ & 2 & 2 \\
$b_i=v'_{\varphi(i)}$ & 3 & 3 & $\mathbf{\underline{2}}$ & 2 & 2 \\
\hline
\multicolumn{6}{|c|}{Group 2 ($i=10,11,\dots,14$)} \\
\hline
$\varphi(i)$          & $\mathbf{\underline{10}}$ & 16 & 22 & 3 & 9 \\
$a_i=u'_{\varphi(i)}$ & $\mathbf{\underline{2}}$ & 4 & 4 & 4 & 4 \\
$b_i=v'_{\varphi(i)}$ & $\mathbf{\underline{4}}$ & 4 & 4 & 4 & 4 \\
\hline
\multicolumn{6}{|c|}{Group 3 ($i=15,16,\dots,19$)} \\
\hline
$\varphi(i)$          & 15 & $\mathbf{\underline{21}}$ & 2 & 8 & 14 \\
$a_i=u'_{\varphi(i)}$ & 4 & $\mathbf{\underline{4}}$ & 1 & 1 & 1 \\
$b_i=v'_{\varphi(i)}$ & 4 & $\mathbf{\underline{1}}$ & 1 & 1 & 1 \\
\hline
\multicolumn{6}{|c|}{Group 4 ($i=20,21,\dots,24$)} \\
\hline
$\varphi(i)$          & 20 & 1 & $\mathbf{\underline{7}}$ & 13 & 19 \\
$a_i=u'_{\varphi(i)}$ & 1 & 1 & $\mathbf{\underline{1}}$ & 0 & 0 \\
$b_i=v'_{\varphi(i)}$ & 1 & 1 & $\mathbf{\underline{0}}$ & 0 & 0 \\
\hline
\end{tabular}
\end{table}

\begin{table}[H]
\centering
\caption{The original order of $u'$ and $v'$, divided into five blocks} \label{table2}
\small
\begin{tabular}{|c|ccccc|ccccc|ccccc|ccccc|ccccc|}
\hline
& \multicolumn{5}{c|}{Block 0} & \multicolumn{5}{c|}{Block 1} & \multicolumn{5}{c|}{Block 2} & \multicolumn{5}{c|}{Block 3} & \multicolumn{5}{c|}{Block 4} \\
\hline
$i$
& $\mathbf{\underline{0}}$ & 1 & 2 & 3 & 4
& 5 & 6 & $\mathbf{\underline{7}}$ & 8 & 9
& $\mathbf{\underline{10}}$ & 11 & 12 & 13 & 14
& 15 & 16 & $\mathbf{\underline{17}}$ & 18 & 19
& 20 & $\mathbf{\underline{21}}$ & 22 & 23 & 24 \\
\hline
$u'_i$
& $\mathbf{\underline{0}}$ & 1 & 1 & 4 & 2
& 3 & 3 & $\mathbf{\underline{1}}$ & 1 & 4
& $\mathbf{\underline{2}}$ & 3 & 3 & 0 & 1
& 4 & 4 & $\mathbf{\underline{3}}$ & 3 & 0
& 1 & $\mathbf{\underline{4}}$ & 4 & 2 & 3 \\
$v'_i$
& $\mathbf{\underline{3}}$ & 1 & 1 & 4 & 2
& 3 & 3 & $\mathbf{\underline{0}}$ & 1 & 4
& $\mathbf{\underline{4}}$ & 3 & 3 & 0 & 1
& 4 & 4 & $\mathbf{\underline{2}}$ & 3 & 0
& 1 & $\mathbf{\underline{1}}$ & 4 & 2 & 3 \\
\hline
\end{tabular}
\end{table}
\end{example}

\subsection{Construction for Multiple-Cycle Case}
We now consider the general case where $\pi \in S_n$ decomposes into $m$ disjoint cycles $c_1, \dots, c_m$ of lengths $\ell_1, \dots, \ell_m$, respectively, where $\sum_{s=1}^m \ell_s = n$. 

We will apply a modified version of the single-cycle construction to each cycle independently, and then combine the results using columnwise interleaving.

\subsubsection{Processing Each Cycle}
For each cycle $c_s$ ($1 \leq s \leq m$), we will construct words $u^{[s]}$ and $v^{[s]}$, each of length $n^2$. We maintain the same block and group structures, as well as the same reading step size $p = n+1$, as defined in the single-cycle case.

Let $\pi^0(x_s), \pi^1(x_s), \dots, \pi^{\ell_s-1}(x_s)$ be the elements of $c_s$, where $x_s$ is an element of $c_s$ with the minimum value. We assign distinguished positions in $u^{[s]}$ and $v^{[s]}$ similarly to the single-cycle case, but for only $\ell_s$ positions. For each $k \in \mathbb{Z}_{\ell_s}$, consider the unique index $f(\pi^k(x_s),\,k)$. We set the letters
\[
u^{[s]}_{f(\pi^k(x_s),\,k)}=\pi^k(x_s) \quad \text{and} \quad
v^{[s]}_{f(\pi^k(x_s),\,k)}=\pi^{k+1}(x_s).
\]

These are distinguished positions in $u^{[s]}$ and $v^{[s]}$. Note that there is at most one distinguished position in each block (and exactly one in each of Groups $0,1,\dots,\ell_s-1$, and none in the remaining groups).

Once the $\ell_s$ distinguished positions are filled, we fill the remaining $n^2 - \ell_s$ positions in the same way as the single-cycle case.

For convenience, define $a^{[s]}_i := u^{[s]}_{\varphi(i)}$ and $b^{[s]}_i := v^{[s]}_{\varphi(i)}$. Let $q$ be an index such that $\varphi(q) = f(x_s,0)$. Then, $\varphi(q)$ is a distinguished position, with $a^{[s]}_q = x_s$ and $b^{[s]}_q = \pi(x_s)$ already filled. We now fill the remaining positions by following the reading order $\varphi(q+1),\varphi(q+2),\dots,\varphi(n^2-1),\varphi(0),\varphi(1),\dots,\varphi(q-1)$, skipping the distinguished positions already filled. Specifically, for each $i=q+1,q+2,\dots,n^2-1,0,1,\dots,q-1$, if $\varphi(i)$ is a non-distinguished position, we set
\[
a^{[s]}_i = b^{[s]}_{i-1}  \quad \text{and} \quad
b^{[s]}_i = b^{[s]}_{i-1}.
\]

We have filled both distinguished and non-distinguished positions and obtain sequences $u^{[s]}$ and $v^{[s]}$ of length $n^2$.

\begin{lemma} \label{main3}
For each cycle $c_s$, the words $R_p(u^{[s]})$ and $R_p(v^{[s]})$ are cyclically equivalent with offset $1$.
\end{lemma}

\begin{proof}
The proof is very similar to that of Lemma~\ref{main1}, with the cycle $\pi^0(0), \pi^1(0),\dots, \pi^{n-1}(0)$ replaced by $\pi^0(x_s),\pi^1(x_s),\dots,\pi^{\ell_s-1}(x_s)$. First, write $R_p(u^{[s]}) = a^{[s]}_0, a^{[s]}_1, \cdots, a^{[s]}_{n^2-1}$ and $R_p(v^{[s]}) = b^{[s]}_0, b^{[s]}_1, \cdots, b^{[s]}_{n^2-1}$. We claim that $b^{[s]}_i = a^{[s]}_{i+1}$ for every $i \in \mathbb{Z}_{n^2}$, where the indices are taken modulo $n^2$.

If $\varphi(i+1)$ is a non-distinguished position, then $a^{[s]}_{i+1} = b^{[s]}_i$ by construction. Now suppose $\varphi(i+1)$ is distinguished. Let it belong to Block $t$ and Group $g$. Then, $t = \pi^g(x_s)$ and $\varphi(i+1) = f(t,g) = f(\pi^g(x_s),g)$. Also, $a^{[s]}_{i+1} = u^{[s]}_{\varphi(i+1)} = \pi^g(x_s)$.

\textbf{Case 1:} $g>0$. Consider the maximum $j$ such that $j<i+1$ and $\varphi(j)$ is distinguished. By the same argument as in Case~1 of the proof of Lemma~\ref{main1}, $\varphi(j)$ is in Group $g-1$ and $b^{[s]}_j = v^{[s]}_{\varphi(j)} = \pi^g(x_s)$. Since $\varphi(j+1), \varphi(j+2), \dots, \varphi(i)$ are non-distinguished, we have $b^{[s]}_i = b^{[s]}_{i-1} = \cdots = b^{[s]}_j = \pi^g(x_s)$. Hence, $a^{[s]}_{i+1} = b^{[s]}_i$.

\textbf{Case 2:} $g=0$. We have $\varphi(i+1) = \varphi(q)$, so $i=q-1$ and $a^{[s]}_{i+1}=x_s$. Consider the maximum $j \in \mathbb{Z}_{n^2}$ such that $\varphi(j)$ is distinguished. By the same argument as in Case~2 of the proof of Lemma~\ref{main1}, $\varphi(j)$ is in Group $\ell_s-1$ and $b^{[s]}_j = v^{[s]}_{\varphi(j)} = x_s$. Since $\varphi(j+1), \varphi(j+2), \dots, \varphi(n^2-1), \varphi(0), \varphi(1), \dots, \varphi(q-1)$ are non-distinguished, we have $b^{[s]}_i = b^{[s]}_{q-1} = b^{[s]}_{q-2} = \cdots = b^{[s]}_0 = b^{[s]}_{n^2-1} = b^{[s]}_{n^2-2} = \dots = b^{[s]}_j = x_s$. Hence, $a^{[s]}_{i+1} = b^{[s]}_i$.

Therefore, $R_p(u^{[s]})$ and $R_p(v^{[s]})$ are cyclically equivalent with offset $1$.
\end{proof}

\subsubsection{Combining All Cycles}
After processing all cycles, we obtain $m$ pairs of cyclically equivalent words $(u^{[1]}, v^{[1]}), \dots,$ $(u^{[m]}, v^{[m]})$, all of the same length $n^2$ and with the same offset $p$. We now apply columnwise interleaving to combine them:
\[
u' = I(u^{[1]}, \dots, u^{[m]}) \quad \text{and} \quad v' = I(v^{[1]}, \dots, v^{[m]}).
\]
Since all pairs $(u^{[s]}, v^{[s]})$ have the same cyclic offset $p$, Lemma~\ref{interleaving} implies that $u'$ and $v'$ are cyclically equivalent.

\begin{lemma} \label{main4}
The words $u'$ and $v'$ can be obtained from $u$ and $v$, respectively, by a simultaneous insertion.
\end{lemma}

\begin{proof}
For each block $t$, there is a unique cycle $c_s$ and a unique integer $k \in \mathbb{Z}_{\ell_s}$ such that $t=\pi^k(x_s)$. Therefore, the only sequences having a distinguished position in Block $t$ are $u^{[s]}$ and $v^{[s]}$. In both sequences, the distinguished position in Block $t$ is $f(t,k)$, and at that position we have
\[
u^{[s]}_{f(t,k)}=t
\quad \text{and} \quad
v^{[s]}_{f(t,k)}=\pi(t).
\]

Columnwise interleaving preserves the relative order of positions within each component word, and thus preserves the order of blocks. Also, non-distinguished positions contain the same letters in $u'$ and $v'$. Therefore, if we keep only the distinguished positions, and delete all other positions, from $u'$ we obtain
\[
0\,1\,\cdots\,(n-1)=u,
\]
while from $v'$ we obtain
\[
\pi(0)\,\pi(1)\,\cdots\,\pi(n-1)=v.
\]
Since the same letters are deleted in both words, $u'$ and $v'$ can be obtained from $u$ and $v$, respectively, by a simultaneous insertion.
\end{proof}

From Lemmas~\ref{step}, \ref{interleaving}, \ref{main3} and~\ref{main4}, we now obtain the desired conclusion for the general case.

\begin{theorem} \label{mainthm2}
For any $\pi \in S_n$, the words
\[
u = 0\,1\,\cdots\,(n-1) \quad \text{and} \quad
v = \pi(0)\,\pi(1)\,\cdots\,\pi(n-1)
\]
are cyclically equalizable.
\end{theorem}

\begin{example}
Let $n=5$, $u=01234$, and $v=34021$. We have $\pi(0)=3$, $\pi(1)=4$, $\pi(2)=0$, $\pi(3)=2$, and $\pi(4)=1$, so $\pi=(0\,3\,2)(1\,4)$ consists of two cycles. The distinguished positions in $u^{[1]}$ and $v^{[1]}$ are
\[
f(0,0)=0,\quad f(3,1)=17,\quad f(2,2)=10,
\] and the distinguished positions in $u^{[2]}$ and $v^{[2]}$ are
\[
f(1,0)=6,\quad f(4,1)=23.
\]

Table~\ref{table3} shows the words $R_p(u^{[1]}), R_p(v^{[1]}), R_p(u^{[2]}), R_p(v^{[2]})$, which are the reading with step size $p=6$ of $u^{[1]}, v^{[1]}, u^{[2]}, v^{[2]}$, respectively. The distinguished positions are shown in bold underlined. One can see that $R_p(v^{[1]})$ (resp. $R_p(v^{[2]})$) is a cyclic shift of $R_p(u^{[1]})$ (resp. $R_p(u^{[2]})$) by one position.

Table~\ref{table4} shows the words $u^{[1]}, v^{[1]}, u^{[2]}, v^{[2]}$ written in their original order. Again, the distinguished positions are shown in bold underlined. 
\begin{table}[H]
\centering
\caption{The reading with step size $p=6$ of $u^{[1]},v^{[1]},u^{[2]},v^{[2]}$} \label{table3}
\small
\begin{tabular}{|c|ccccc|}
\hline
\multicolumn{6}{|c|}{Group 0 ($i=0,1,\dots,4$)} \\
\hline
$\varphi(i)$          & $\mathbf{\underline{0}}$ & 6 & 12 & 18 & 24 \\
$a^{[1]}_i=u^{[1]}_{\varphi(i)}$ & $\mathbf{\underline{0}}$ & 3 & 3 & 3 & 3 \\
$b^{[1]}_i=v^{[1]}_{\varphi(i)}$ & $\mathbf{\underline{3}}$ & 3 & 3 & 3 & 3 \\
\hline
\multicolumn{6}{|c|}{Group 1 ($i=5,6,\dots,9$)} \\
\hline
$\varphi(i)$          & 5 & 11 & $\mathbf{\underline{17}}$ & 23 & 4 \\
$a^{[1]}_i=u^{[1]}_{\varphi(i)}$ & 3 & 3 & $\mathbf{\underline{3}}$ & 2 & 2 \\
$b^{[1]}_i=v^{[1]}_{\varphi(i)}$ & 3 & 3 & $\mathbf{\underline{2}}$ & 2 & 2 \\
\hline
\multicolumn{6}{|c|}{Group 2 ($i=10,11,\dots,14$)} \\
\hline
$\varphi(i)$          & $\mathbf{\underline{10}}$ & 16 & 22 & 3 & 9 \\
$a^{[1]}_i=u^{[1]}_{\varphi(i)}$ & $\mathbf{\underline{2}}$ & 0 & 0 & 0 & 0 \\
$b^{[1]}_i=v^{[1]}_{\varphi(i)}$ & $\mathbf{\underline{0}}$ & 0 & 0 & 0 & 0 \\
\hline
\multicolumn{6}{|c|}{Group 3 ($i=15,16,\dots,19$)} \\
\hline
$\varphi(i)$          & 15 & 21 & 2 & 8 & 14 \\
$a^{[1]}_i=u^{[1]}_{\varphi(i)}$ & 0 & 0 & 0 & 0 & 0 \\
$b^{[1]}_i=v^{[1]}_{\varphi(i)}$ & 0 & 0 & 0 & 0 & 0 \\
\hline
\multicolumn{6}{|c|}{Group 4 ($i=20,21,\dots,24$)} \\
\hline
$\varphi(i)$          & 20 & 1 & 7 & 13 & 19 \\
$a^{[1]}_i=u^{[1]}_{\varphi(i)}$ & 0 & 0 & 0 & 0 & 0 \\
$b^{[1]}_i=v^{[1]}_{\varphi(i)}$ & 0 & 0 & 0 & 0 & 0 \\
\hline
\end{tabular} \hspace{1cm}
\begin{tabular}{|c|ccccc|}
\hline
\multicolumn{6}{|c|}{Group 0 ($i=0,1,\dots,4$)} \\
\hline
$\varphi(i)$          & 0 & $\mathbf{\underline{6}}$ & 12 & 18 & 24 \\
$a^{[2]}_i=u^{[2]}_{\varphi(i)}$ & 1 & $\mathbf{\underline{1}}$ & 4 & 4 & 4 \\
$b^{[2]}_i=v^{[2]}_{\varphi(i)}$ & 1 & $\mathbf{\underline{4}}$ & 4 & 4 & 4 \\
\hline
\multicolumn{6}{|c|}{Group 1 ($i=5,6,\dots,9$)} \\
\hline
$\varphi(i)$          & 5 & 11 & 17 & $\mathbf{\underline{23}}$ & 4 \\
$a^{[2]}_i=u^{[2]}_{\varphi(i)}$ & 4 & 4 & 4 & $\mathbf{\underline{4}}$ & 1 \\
$b^{[2]}_i=v^{[2]}_{\varphi(i)}$ & 4 & 4 & 4 & $\mathbf{\underline{1}}$ & 1 \\
\hline
\multicolumn{6}{|c|}{Group 2 ($i=10,11,\dots,14$)} \\
\hline
$\varphi(i)$          & 10 & 16 & 22 & 3 & 9 \\
$a^{[2]}_i=u^{[2]}_{\varphi(i)}$ & 1 & 1 & 1 & 1 & 1 \\
$b^{[2]}_i=v^{[2]}_{\varphi(i)}$ & 1 & 1 & 1 & 1 & 1 \\
\hline
\multicolumn{6}{|c|}{Group 3 ($i=15,16,\dots,19$)} \\
\hline
$\varphi(i)$          & 15 & 21 & 2 & 8 & 14 \\
$a^{[2]}_i=u^{[2]}_{\varphi(i)}$ & 1 & 1 & 1 & 1 & 1 \\
$b^{[2]}_i=v^{[2]}_{\varphi(i)}$ & 1 & 1 & 1 & 1 & 1 \\
\hline
\multicolumn{6}{|c|}{Group 4 ($i=20,21,\dots,24$)} \\
\hline
$\varphi(i)$          & 20 & 1 & 7 & 13 & 19 \\
$a^{[2]}_i=u^{[2]}_{\varphi(i)}$ & 1 & 1 & 1 & 1 & 1 \\
$b^{[2]}_i=v^{[2]}_{\varphi(i)}$ & 1 & 1 & 1 & 1 & 1 \\
\hline
\end{tabular}
\end{table}

\begin{table}[H]
\centering
\caption{The original order of $u^{[1]},v^{[1]},u^{[2]},v^{[2]}$} \label{table4}
\small
\begin{tabular}{cccccc|ccccc|ccccc|ccccc|ccccc|}
\hline
\multicolumn{1}{|c|}{} & \multicolumn{5}{c|}{Block 0} & \multicolumn{5}{c|}{Block 1} & \multicolumn{5}{c|}{Block 2} & \multicolumn{5}{c|}{Block 3} & \multicolumn{5}{c|}{Block 4} \\
\hline
\multicolumn{1}{|c|}{$i$}
& $\mathbf{\underline{0}}$ & 1 & 2 & 3 & 4
& 5 & 6 & 7 & 8 & 9
& $\mathbf{\underline{10}}$ & 11 & 12 & 13 & 14
& 15 & 16 & $\mathbf{\underline{17}}$ & 18 & 19
& 20 & 21 & 22 & 23 & 24 \\
\hline
\multicolumn{1}{|c|}{$u^{[1]}_i$}
& $\mathbf{\underline{0}}$ & 0 & 0 & 0 & 2
& 3 & 3 & 0 & 0 & 0
& $\mathbf{\underline{2}}$ & 3 & 3 & 0 & 0
& 0 & 0 & $\mathbf{\underline{3}}$ & 3 & 0
& 0 & 0 & 0 & 2 & 3 \\
\multicolumn{1}{|c|}{$v^{[1]}_i$}
& $\mathbf{\underline{3}}$ & 0 & 0 & 0 & 2
& 3 & 3 & 0 & 0 & 0
& $\mathbf{\underline{0}}$ & 3 & 3 & 0 & 0
& 0 & 0 & $\mathbf{\underline{2}}$ & 3 & 0
& 0 & 0 & 0 & 2 & 3 \\
\hline \\
\end{tabular}

\begin{tabular}{cccccc|ccccc|ccccc|ccccc|ccccc|}
\hline
\multicolumn{1}{|c|}{} & \multicolumn{5}{c|}{Block 0} & \multicolumn{5}{c|}{Block 1} & \multicolumn{5}{c|}{Block 2} & \multicolumn{5}{c|}{Block 3} & \multicolumn{5}{c|}{Block 4} \\
\hline
\multicolumn{1}{|c|}{$i$}
& 0 & 1 & 2 & 3 & 4
& 5 & $\mathbf{\underline{6}}$ & 7 & 8 & 9
& 10 & 11 & 12 & 13 & 14
& 15 & 16 & 17 & 18 & 19
& 20 & 21 & 22 & $\mathbf{\underline{23}}$ & 24 \\
\hline
\multicolumn{1}{|c|}{$u^{[2]}_i$}
& 1 & 1 & 1 & 1 & 1
& 4 & $\mathbf{\underline{1}}$ & 1 & 1 & 1
& 1 & 4 & 4 & 1 & 1
& 1 & 1 & 4 & 4 & 1
& 1 & 1 & 1 & $\mathbf{\underline{4}}$ & 4 \\
\multicolumn{1}{|c|}{$v^{[2]}_i$}
& 1 & 1 & 1 & 1 & 1
& 4 & $\mathbf{\underline{4}}$ & 1 & 1 & 1
& 1 & 4 & 4 & 1 & 1
& 1 & 1 & 4 & 4 & 1
& 1 & 1 & 1 & $\mathbf{\underline{1}}$ & 4 \\
\hline
\end{tabular}
\end{table}
Finally, we construct columnwise interleaved words $u' = I(u^{[1]},u^{[2]})$ and $v' = I(v^{[1]},v^{[2]})$, each of length 50, that are cyclically equivalent to each other. Deleting all non-bold entries yields $u$ from $u'$ and $v$ from $v'$.
\begin{align*}
u' &=\h01\;  01\;  01\;  01\;  21\mid 
       34\;3\h1\;  01\;  01\;  01\mid 
     \h21\;  34\;  34\;  01\;  01\mid 
       01\;  01\;\h34\;  34\;  01\mid 
       01\;  01\;  01\;2\h4\;  34, \\
v' &= \h31\;  01\;  01\;  01\;  21\mid 
        34\;3\h4\;  01\;  01\;  01\mid 
      \h01\;  34\;  34\;  01\;  01\mid 
        01\;  01\;\h24\;  34\;  01\mid 
        01\;  01\;  01\;2\h1\;  34.
\end{align*}
\end{example}

\section{Consequences and Interpretations}
In this section, we discuss immediate consequences of our main theorem and its connections to existing notions in combinatorics on words.

\subsection{Algorithmic Implications}
The main theorem yields a simple decision procedure for cyclic equalizability: checking whether the Parikh vectors of two words are equal can be done in $O(n)$ time, where $n$ is the length of the words.

Moreover, the proof of the main theorem provides a constructive method to transform two Abelian equivalent words into cyclically equivalent words via simultaneous insertions. This yields an explicit algorithm for finding such insertions in $O(mn^2)$ time, where $m$ is the number of cycles in $\pi$.

\subsection{Relation to Circular Words}
Cyclic equalizability can be viewed as a bridge between Abelian equivalence and cyclic equivalence. Our result shows that, for any two Abelian equivalent words, there exist superwords obtained by inserting letters at the same positions such that the resulting words are conjugate. This provides a link between these two classical notions.

\section{Applications to Card-Based Cryptography}
We briefly discuss implications of our result for card-based cryptography. In this setting, a sequence of face-down cards can be modeled as a word over a finite alphabet, where each letter represents a type of card. A fundamental operation in many protocols is the \emph{random cut}, which cyclically shifts the sequence by an unknown offset. Because of this operation, two sequences that differ only by a cyclic shift can be made indistinguishable to all players.

Cyclic equalizability was introduced by Shinagawa and Nuida to capture this phenomenon. In their framework, inserting letters corresponds to inserting additional cards into the sequence without revealing the existing cards. Our characterization simplifies the analysis of such protocols. In particular, our main theorem shows that two sequences can be made indistinguishable under random cuts, after inserting additional cards, if and only if they have the same multiset of card types.

Our result can also be interpreted as a statement about \emph{information erasure}. Let $S \subseteq \Sigma^n$ be a set of possible inputs, and suppose a sequence of face-down cards $X=(x_1,\dots,x_n) \in S$ is given, but its exact value is unknown. A protocol achieves information erasure if, after inserting additional cards and applying a random cut, the distribution of the revealed output does not depend on the choice of $X \in S$. Our result implies that information erasure is always achievable in the case $|S|=2$, since any two sequences with the same multiset of cards can be made cyclically indistinguishable by suitable insertions.

\section{Open Problems}
While this paper provides a complete characterization of cyclic equalizability for two words over arbitrary finite alphabets, several natural questions remain open.

\begin{itemize}
    \item \textbf{Cyclic equalizability for multiple words.}  
    The most immediate open problem is to characterize cyclic equalizability for $k \ge 3$ words. In contrast to the two-word case, it is not clear whether equality of Parikh vectors is sufficient when more than two words are involved. Determining necessary and sufficient conditions in this setting remains an important direction for future work.

    \item \textbf{Minimal insertion length.}  
    Given two cyclically equalizable words, what is the minimum number of letters that must be inserted to obtain cyclically equivalent words? Establishing tight bounds or efficient algorithms for this quantity would strengthen the constructive aspect of cyclic equalizability.

    \item \textbf{Restricted insertions.}  
    One may consider variants of cyclic equalizability where insertions are restricted to a fixed subset of the alphabet. Understanding how such restrictions affect the characterization is an interesting problem.
\end{itemize}

Several of these questions were also raised by Shinagawa and Nuida \cite{cyclic}, particularly for binary words.

\section{Conclusion}
In this paper, we studied cyclic equalizability, a notion introduced in the context of card-based cryptography, from the perspective of combinatorics on words. We established a complete characterization for two words over arbitrary finite alphabets, showing that cyclic equalizability coincides exactly with Abelian equivalence. This result generalizes the previously known characterization for binary words and resolves, as a special case, the question of cyclic equalizability over larger alphabets such as the ternary alphabet.

Beyond its theoretical interest, this characterization has implications for the analysis of card-based cryptographic protocols, where cyclic equalizability models the indistinguishability of sequences under card insertions and random cuts. Our result shows that, in the two-sequence case, this indistinguishability is completely determined by the multiset of cards.

\subsubsection*{Acknowledgement}
The authors would like to thank Kazumasa Shinagawa, Pakawut Jiradilok, Teetat Thamronglak, Weerawat Wongmanit, and Jirapat Tabtimthai for valuable discussions on this research. This work was supported by the 111th Anniversary Engineering Research Catalyst Fund Towards U Top~100.

\end{document}